\documentclass[11pt,reqno]{amsart}
\usepackage[T1]{fontenc}
\usepackage[utf8]{inputenc}
\usepackage{lmodern}
\usepackage{microtype}
\usepackage{amsmath,amssymb,amsthm}
\usepackage{xcolor}
\usepackage{cite}
\usepackage[
colorlinks=true,
citecolor=blue,
linkcolor=blue,
urlcolor=blue
]{hyperref}

\numberwithin{equation}{section}
\newtheorem{theorem}{Theorem}[section]
\newtheorem{proposition}[theorem]{Proposition}
\newtheorem{lemma}[theorem]{Lemma}
\theoremstyle{definition}

\newtheorem{remark}[theorem]{Remark}

\title[Mass-Capacity Rigidity]
{Mass-Capacity Rigidity for Asymptotically Flat Manifolds with Arbitrary Ends}
\author[Y. Bi]{Yuchen Bi}
\address[Yuchen Bi]{Mathematical Institute, Department of Pure Mathematics,
University of Freiburg, Ernst-Zermelo-Stra{\ss}e 1,
D-79104 Freiburg im Breisgau, Germany}
\email{yuchen.bi@math.uni-freiburg.de}
\hypersetup{
pdftitle={Mass-Capacity Rigidity for Asymptotically Flat Manifolds with Arbitrary Ends},
pdfauthor={Yuchen Bi}
}



\begin{document}

\begin{abstract}
We characterize equality in the mass-capacity inequality for asymptotically flat manifolds with arbitrary ends in every dimension \(n\ge 3\). If \(u\) is the capacity minimizer, then the equality forces \((M,u^{4/(n-2)}g)\) to be isometric to \((\mathbb R^n\setminus S, g_{\mathrm{Euc}})\) for a compact set \(S\) satisfying \(\operatorname{dim}_{\mathcal H}S\le (n-2)/2\). A key ingredient is that conformal changes by positive harmonic functions transform nonnegative Ricci curvature into nonnegative Bakry--\'Emery Ricci curvature of effective dimension \(4-n\).
\end{abstract}

\maketitle

\section{Introduction}
Mass-capacity inequalities at a distinguished asymptotically flat end go back to Bray~\cite[Theorem~8]{Bray}, who proved the three-dimensional inequality when all the remaining ends are asymptotically flat. Related inequalities for asymptotically flat three-manifolds with boundary were obtained by Bray and Miao~\cite{BrayMiao}, Hirsch, Miao, and Tam~\cite{HirschMiaoTam}, and Miao~\cite{MiaoMassCapacity}. Under additional topological assumptions and a lower Ricci curvature bound, Miao~\cite[Corollary~3.1]{MiaoImplications} later proved the inequality for manifolds whose other end need not be asymptotically flat.

With no asymptotic conditions imposed on the other ends, the mass--capacity inequality for a distinguished asymptotically flat end was deduced in \cite{BZ} from the corresponding positive mass theorem. Rigidity was established there when the manifold is spin. The positive mass theorem in this setting was proved for \(3\le n\le 7\) in \cite{LLU,LUY,Zhu}. The density results \cite[Theorem~1.3]{LLU}, \cite[Proposition~3.2]{Zhu}, together with the positive mass theorem of Brendle--Wang~\cite{BW}, extend this theorem to every dimension \(n\ge 3\). Here we remove the spin assumption from the rigidity statement in \cite{BZ}.

Let \((M^n,g)\), \(n\ge 3\), be a connected complete Riemannian manifold without boundary. Fix a compact set \(K\subset M\) and a component \(E\) of \(M\setminus K\). We call \((M,g,E)\) an \emph{asymptotically flat manifold with arbitrary ends} if there exist \(R>0\), \(\tau>(n-2)/2\), and a diffeomorphism
\[
x=(x^1,\ldots,x^n):E\longrightarrow\mathbb R^n\setminus\overline{B_R}
\]
such that, writing \(r=|x|\),
\[
g_{ij}-\delta_{ij}=O_2(r^{-\tau}),\qquad R_g\in L^1(E).
\]
We call such \(x\) an asymptotically flat coordinate system on \(E\). No asymptotic condition is imposed on the other ends.

We write \(f=O_k(r^\mu)\) if every coordinate derivative of order \(j\), \(0\le j\le k\), is \(O(r^{\mu-j})\), and use \(o_k(r^\mu)\) analogously. Let \(S_r\subset E\) be the coordinate sphere of radius \(r\), and let \(\nu\) and \(d\sigma\) denote its outward unit normal and area measure with respect to the Euclidean metric. The ADM mass of \(E\) is
\[
m_E(g):=\frac{1}{2(n-1)|\mathbb S^{n-1}|}\lim_{r\to\infty}\int_{S_r}(\partial_jg_{ij}-\partial_ig_{jj})\nu^i d\sigma.
\]
Fix \(\eta_E\in C^\infty(M)\) that, outside a compact set, equals \(1\) on \(E\) and \(0\) on every other end. Let \(\mathcal D_0(g)\) be the completion of \(C_c^\infty(M)\) with respect to the norm
\[
\|\varphi\|_{\mathcal D_0(g)}:=\left(\int_M|\nabla_g\varphi|_g^2d\mu_g\right)^{1/2},
\]
and set \(\mathcal A_E:=\eta_E+\mathcal D_0(g)\). The affine space \(\mathcal A_E\) is independent of the choice of \(\eta_E\). The capacity of \(E\) is
\begin{equation}\label{eq:capacity-definition}
 \mathfrak c_E(g):=\frac{1}{(n-2)|\mathbb S^{n-1}|} \inf_{\psi\in\mathcal A_E}\int_M|\nabla_g\psi|_g^2d\mu_g.  
\end{equation}
The infimum is attained by a unique \(u\in\mathcal A_E\). The function \(u\) is positive and harmonic. On \(E\), it has the asymptotic expansion
\[
u=1-\mathfrak c_E(g)r^{2-n}+o_2(r^{2-n})
\]
as \(r\to\infty\); see, e.g., \cite[Corollary~2.2]{BZ}.

We state the mass--capacity inequality proved in \cite{BZ} together with our rigidity result.

\begin{theorem}\label{thm:main}
Let \((M^n,g,E)\), \(n\ge 3\), be an asymptotically flat manifold with arbitrary ends, and let \(u\) be the minimizer in \eqref{eq:capacity-definition}. If \(R_g\ge 0\), then
\[
m_E(g)\ge 2\mathfrak c_E(g).
\]
Moreover, if equality holds, then \((M,u^{4/(n-2)}g)\) is isometric to \((\mathbb R^n\setminus S, g_{\mathrm{Euc}})\) for some bounded closed set \(S\subset\mathbb R^n\) satisfying 
\[
\operatorname{Cap}_{1+\frac{2}{n},\frac{n}{2}}(S)=0,\qquad\dim_{\mathcal H}S\le\frac{n-2}{2}.
\]
\end{theorem}
Here, for \(\alpha>0\) and \(1<p<\infty\), \(\operatorname{Cap}_{\alpha,p}\) denotes the Bessel capacity
\[
\operatorname{Cap}_{\alpha,p}(S):=\inf\left\{\|f\|_{L^p(\mathbb R^n)}^p:f\in L^p(\mathbb R^n), f\ge 0, G_\alpha*f\ge 1\text{ on }S\right\},
\]
where \(G_\alpha\) is the inverse Fourier transform of \((1+|\xi|^2)^{-\alpha/2}\). If \(\alpha p<n\), then \(\operatorname{Cap}_{\alpha,p}(S)=0\) implies \(\dim_{\mathcal H}S\le n-\alpha p\) \cite[Corollary~5.1.14]{AdamsHedberg}.

Set \(\bar g:=u^{4/(n-2)}g\). In the equality case, a conformal deformation argument shows that \(\bar g\) is Ricci-flat and \(m_E(\bar g)=0\). The metric \(\bar g\) need not be complete. In dimension three, Ricci-flatness implies flatness. For \(n\ge 4\), this implication fails. For \(0<t<1\), set
\[
g_t:=\left((1-t)+tu^{-1}\right)^{\frac{4}{n-2}}\bar g,\qquad F_t:=2\log\left((1-t)+tu^{-1}\right).
\]
Since \(g_t\ge t^{4/(n-2)}g\), each \(g_t\) is complete. Moreover, \(g_t\to\bar g\) locally smoothly as \(t\to 0\), while
\[
\operatorname{Ric}_{F_t}^{4-n}(g_t)\ge 0,\qquad m_E(g_t)=2t\mathfrak c_E(g).
\]
The use of \(4-n\) is motivated by the observation in recent work \cite{BZpsc} that positivity of weighted scalar curvature becomes easier to preserve under conformal changes as the reciprocal  of the effective dimension \(N\) decreases, provided that \(1/N\le 1/n\).

We solve a weighted harmonic equation for functions \(X_t^1,\ldots,X_t^n\) asymptotic to the Euclidean coordinate functions. Integrating a weighted Bochner formula gives an \(L^2\) estimate for their Hessians in terms of \(m_E(g_t)\). Since \(m_E(g_t)\to 0\), the limiting functions therefore define a local isometry
\[
X:(M,\bar g)\longrightarrow\mathbb R^n.
\]
In dimension three, one instead takes \(X\) to be a developing map for the flat metric \(\bar g\). In either case, the Schoen--Yau Liouville theorem~\cite[Theorem~1.3]{LUY} shows that \(X\) is injective. The complement of \(X(M)\) is a bounded closed set \(S\), and \(X\) identifies \((M,\bar g)\)  with \((\mathbb R^n\setminus S,g_{\mathrm{Euc}})\). Since \(g\) is complete and scalar-flat, Karakhanyan's theorem~\cite[Theorem~1.1]{Karakhanyan} gives \(\operatorname{Cap}_{1+2/n,n/2}(S)=0\).

\section{Conformal deformation and Ricci flatness}

\subsection{Mass and the conformal Laplacian}
Recall that \(\mathcal A_E=\eta_E+\mathcal D_0(g)\). Set \(c_n:=4(n-1)/(n-2)\) and \(Q_g(\psi):=\int_M(c_n|\nabla_g\psi|_g^2+R_g\psi^2)d\mu_g\). For every positive function \(\phi\),
\begin{equation}\label{eq:conformal-scalar-curvature}
R_{\phi^{4/(n-2)}g}=\phi^{-\frac{n+2}{n-2}}(-c_n\Delta_g\phi+R_g\phi).
\end{equation}
If \(\phi=1+br^{2-n}+o_1(r^{2-n})\) on \(E\), then
\begin{equation}\label{eq:conformal-mass}
m_E\left(\phi^{4/(n-2)}g\right)=m_E(g)+2b.
\end{equation}

\begin{proposition}\label{prop:conformal-mass-bound}
If \(R_g\ge 0\), then
\[
m_E(g)\ge\frac{1}{2(n-1)|\mathbb S^{n-1}|}\inf_{\psi\in\mathcal A_E}Q_g(\psi)\ge 2\mathfrak c_E(g).
\]
\end{proposition}
\begin{proof}
Fix \(\rho\in C_c^\infty(M)\) with \(0\le\rho\le R_g\), and set \(Q_\rho(\psi):=\int_M(c_n|\nabla_g\psi|_g^2+\rho\psi^2)d\mu_g\). Writing \(\psi=\eta_E+v\) with \(v\in\mathcal D_0(g)\), we have
\[
Q_\rho(\eta_E+v)\ge\frac{c_n}{2}\|v\|_{\mathcal D_0(g)}^2-c_n\int_M|\nabla_g\eta_E|_g^2d\mu_g.
\]
Thus \(Q_\rho\) has a minimizer \(w\in\mathcal A_E\). By truncation, we may assume \(0\le w\le 1\). Its Euler--Lagrange equation is \(-c_n\Delta_gw+\rho w=0\). 

Since \(\rho\) is compactly supported, \(1-w\) is harmonic outside a compact set, so  \cite[Lemma~2.1]{BZ} gives \(w=1-ar^{2-n}+o_2(r^{2-n})\) on \(E\) for some \(a\in\mathbb R\). With \(w-\eta_E\) as test function, the Euler--Lagrange equation gives
\[
Q_\rho(w)=\int_M\left(c_n\langle\nabla_gw,\nabla_g\eta_E\rangle_g+\rho w\eta_E\right)d\mu_g=c_n\lim_{r\to\infty}\int_{S_r}\partial_{\nu_g}wd\sigma_g=4(n-1)|\mathbb S^{n-1}|a.
\]
For \(\varepsilon>0\), let \(\phi_\varepsilon=(w+\varepsilon)/(1+\varepsilon)\) and \(g_\varepsilon=\phi_\varepsilon^{4/(n-2)}g\). This metric is complete and asymptotically flat. \eqref{eq:conformal-scalar-curvature} and \eqref{eq:conformal-mass} give
\[
R_{g_\varepsilon}=\phi_\varepsilon^{-\frac{n+2}{n-2}}\frac{(R_g-\rho)w+\varepsilon R_g}{1+\varepsilon}\ge 0,\qquad m_E(g_\varepsilon)=m_E(g)-\frac{2a}{1+\varepsilon}.
\]
The positive mass theorem with arbitrary ends applies to \(g_\varepsilon\), so \(m_E(g_\varepsilon)\ge 0\). Letting \(\varepsilon\to 0\), we obtain 
\[
m_E(g)\ge 2a=\frac{1}{2(n-1)|\mathbb S^{n-1}|}\inf_{\mathcal A_E}Q_\rho.
\]
Choose \(\rho_j=\chi_jR_g\), where \(0\le\chi_j\uparrow 1\) and \(\chi_j\in C_c^\infty(M)\), and let \(w_j\) minimize \(Q_{\rho_j}\). Since \(Q_{\rho_j}(w_j)\le Q_g(\eta_E)\), the coercivity bound for \(Q_{\rho_j}(\eta_E+v)\), applied with \(v=w_j-\eta_E\), shows that \(w_j-\eta_E\) is bounded in \(\mathcal D_0(g)\). After passing to a subsequence, \(w_j-\eta_E\rightharpoonup w-\eta_E\) in \(\mathcal D_0(g)\) for some \(w\in\mathcal A_E\). For each fixed \(k\), since \(\rho_k\le\rho_j\) for \(j\ge k\),
\[
Q_{\rho_k}(w)\le\liminf_{j\to\infty}Q_{\rho_k}(w_j)\le\liminf_{j\to\infty}Q_{\rho_j}(w_j)
\]
Letting \(k\to\infty\), we obtain
\[
\inf_{\mathcal A_E}Q_g\le Q_g(w)\le\lim_{j\to\infty}\inf_{\mathcal A_E}Q_{\rho_j}\le\inf_{\mathcal A_E}Q_g.
\]
This proves the first inequality. The second inequality follows immediately from \(R_g\ge 0\) and the definition of \(\mathfrak c_E(g)\).
\end{proof}

The following nonparabolicity estimate will be used to extend the first inequality in Proposition~\ref{prop:conformal-mass-bound} to small compactly supported perturbations of a scalar-flat metric.

\begin{lemma}\label{lem:local-Poincare}
Let \(h\) be a smooth metric on \(M\) that is asymptotically flat at the distinguished end \(E\). For every compact set \(K'\subset M\), there exists \(C_{K'}>0\) such that every \(\varphi\in C_c^\infty(M)\) satisfies
\[
\|\varphi\|_{L^{\frac{2n}{n-2}}(K',h)}\le C_{K'}\|\varphi\|_{\mathcal D_0(h)}.
\]
\end{lemma}
\begin{proof}
    Choose \(R\) sufficiently large that \(A:=\{R<r<2R\}\Subset E\) and \(h\) is uniformly equivalent to the Euclidean metric on \(\{r>R\}\). Since \(\varphi\) is compactly supported,
    \[
    |\varphi(r,\theta)|^2=\left|\int_r^\infty\partial_s\varphi(s,\theta)ds\right|^2\le\frac{r^{2-n}}{n-2}\int_r^\infty|\partial_s\varphi(s,\theta)|^2s^{n-1}ds.
    \]
    Integrating over \(A\) gives \(\|\varphi\|_{L^2(A,h)}\le C\|\varphi\|_{\mathcal D_0(h)}\). Now choose a connected smooth domain \(\Omega\Subset M\) containing \(K'\cup\overline A\). Sobolev-Poincar\'e inequality on \(\Omega\) gives
    \[
    \|\varphi\|_{L^{\frac{2n}{n-2}}(K',h)}\le C_\Omega\left(\|\nabla_h\varphi\|_{L^2(\Omega,h)}+\|\varphi\|_{L^2(A,h)}\right)\le C_{K'}\|\varphi\|_{\mathcal D_0(h)}.
    \]
\end{proof}

\begin{lemma}\label{lem:perturbative-conformal-mass-bound}
Let \((M,g,E)\) be asymptotically flat with arbitrary ends and suppose that \(R_g=0\). Fix a compact set \(K\subset M\). If a metric \(\gamma\) equals \(g\) on \(M\setminus K\) and is sufficiently close to \(g\) in \(C^2(K)\), then 
\[
m_E(\gamma)\ge\frac{1}{2(n-1)|\mathbb S^{n-1}|}\inf_{\psi\in\mathcal A_E}Q_\gamma(\psi).
\]
\end{lemma}
\begin{proof}
Since \(R_g=0\) and \(\gamma=g\) outside \(K\), the scalar curvature \(R_\gamma\) is supported in \(K\). Lemma~\ref{lem:local-Poincare} gives \(\|v\|_{L^{\frac{2n}{n-2}}(K,\gamma)}\le C\|v\|_{\mathcal D_0(\gamma)}\) for \(v\in\mathcal D_0(\gamma)\). As \(\gamma\to g\) in \(C^2(K)\), \(\|R_\gamma^-\|_{L^{n/2}(K,\gamma)}=o(1)\), and hence
\begin{equation}\label{eq:perturbative-coercivity}
Q_\gamma(v)\ge(c_n-o(1))\|v\|_{\mathcal D_0(\gamma)}^2.
\end{equation}
Let \(w\in\mathcal A_E\) be the minimizer of \(Q_\gamma\). Its Euler--Lagrange equation is \(-c_n\Delta_\gamma w+R_\gamma w=0\). As in the proof of Proposition~\ref{prop:conformal-mass-bound}, for some \(a\in\mathbb R\),
\[
w=1-ar^{2-n}+o_2(r^{2-n})\quad\text{on }E,\qquad Q_\gamma(w)=4(n-1)|\mathbb S^{n-1}|a.
\]
For \(\varepsilon>0\), let \(\phi_\varepsilon:=(w+\varepsilon)/(1+\varepsilon)\) and \(\gamma_\varepsilon=\phi_\varepsilon^{4/(n-2)}\gamma\). The metric \(\gamma_\varepsilon\) is complete and asymptotically flat. \eqref{eq:conformal-scalar-curvature} and \eqref{eq:conformal-mass} give
\[
R_{\gamma_\varepsilon}=\frac{\varepsilon}{1+\varepsilon}\phi_\varepsilon^{-\frac{n+2}{n-2}}R_\gamma,\qquad m_E(\gamma_\varepsilon)=m_E(\gamma)-\frac{2a}{1+\varepsilon}.
\]
Put \(q_\varepsilon:=\max\{-R_{\gamma_\varepsilon},0\}\). Then \(\operatorname{supp}q_\varepsilon\subset K\) and \(\|q_\varepsilon\|_{L^{n/2}(K,\gamma_\varepsilon)}=O(\varepsilon)\). Then for \(\varphi\in\mathcal D_0(\gamma_\varepsilon)\), we know
\[
\begin{aligned}
    J_\varepsilon(\varphi):=&\int_M\left(c_n|\nabla_{\gamma_\varepsilon}\varphi|_{\gamma_\varepsilon}^2-q_\varepsilon\varphi^2-2q_\varepsilon\varphi\right)d\mu_{\gamma_\varepsilon}\\
    \ge&(c_n-C\varepsilon)\|\varphi\|_{\mathcal D_0(\gamma_\varepsilon)}^2-C\varepsilon\|\varphi\|_{\mathcal D_0(\gamma_\varepsilon)}\\
    \ge&\frac{c_n}{2}\|\varphi\|_{\mathcal D_0(\gamma_\varepsilon)}^2-C\varepsilon^2.
\end{aligned}
\]
So \(J_\varepsilon\) has a unique minimizer \(\theta_\varepsilon\in\mathcal D_0(\gamma_\varepsilon)\), whose Euler--Lagrange equation is
\[
(-c_n\Delta_{\gamma_\varepsilon}-q_\varepsilon)\theta_\varepsilon=q_\varepsilon.
\]
Since \(J_\varepsilon(\theta_\varepsilon)\le J_\varepsilon(0)=0\), the bound above gives \(\|\theta_\varepsilon\|_{\mathcal D_0(\gamma_\varepsilon)}=O(\varepsilon)\). The maximum principle gives \(\theta_\varepsilon\ge 0\).

At \(E\), write \(\theta_\varepsilon=b_\varepsilon r^{2-n}+o_2(r^{2-n})\). With \(\eta_E\) as the test function, the Euler--Lagrange equation gives
\[
-4(n-1)|\mathbb S^{n-1}|b_\varepsilon=c_n\int_M\langle\nabla_{\gamma_\varepsilon}\theta_\varepsilon,\nabla_{\gamma_\varepsilon}\eta_E\rangle_{\gamma_\varepsilon}d\mu_{\gamma_\varepsilon}-\int_Mq_\varepsilon(1+\theta_\varepsilon)\eta_Ed\mu_{\gamma_\varepsilon}.
\]
Both terms on the right are \(O(\varepsilon)\), so \(b_\varepsilon=O(\varepsilon)\).

Finally, let \(\widetilde\gamma_\varepsilon=(1+\theta_\varepsilon)^{4/(n-2)}\gamma_\varepsilon\). This metric is complete and asymptotically flat, with
\[
R_{\widetilde\gamma_\varepsilon}=(1+\theta_\varepsilon)^{-\frac{4}{n-2}}(R_{\gamma_\varepsilon}+q_\varepsilon),\qquad m_E(\widetilde\gamma_\varepsilon)=m_E(\gamma)-\frac{2a}{1+\varepsilon}+2b_\varepsilon.
\]
The positive mass theorem with arbitrary ends gives \(m_E(\widetilde\gamma_\varepsilon)\ge 0\). Letting \(\varepsilon\to 0\), we obtain
\[
m_E(\gamma)\ge 2a=\frac{1}{2(n-1)|\mathbb S^{n-1}|}\inf_{\psi\in\mathcal A_E}Q_\gamma(\psi).
\]
\end{proof}

\subsection{Scalar and Ricci flatness}

\begin{proposition}\label{prop:equality-reduction}
Assume that \((M^n,g,E)\) has nonnegative scalar curvature and \(m_E(g)=2\mathfrak c_E(g)\). Let \(u\) be the minimizer in \eqref{eq:capacity-definition}, and set \(\bar g:=u^{4/(n-2)}g\) and \(V:=u^{-1}\). Then
\[
R_g=0,\qquad m_E(\bar{g})=0,\qquad \operatorname{Ric}_{\bar g}=0,\qquad \Delta_{\bar g}V=0.
\]
\end{proposition}
\begin{proof}
By assumption and Proposition~\ref{prop:conformal-mass-bound},
\[
4(n-1)|\mathbb S^{n-1}|\mathfrak c_E(g)\le\inf_{\psi\in\mathcal A_E}Q_g(\psi)\le 2(n-1)|\mathbb S^{n-1}|m_E(g)=4(n-1)|\mathbb S^{n-1}|\mathfrak c_E(g).
\]
Let \(w\) minimize \(Q_g\) on \(\mathcal A_E\). Since \(w-u\in\mathcal D_0(g)\) and \(u\) is harmonic,
\[
0=Q_g(w)-4(n-1)|\mathbb S^{n-1}|\mathfrak c_E(g)=c_n\int_M|\nabla_g(w-u)|_g^2d\mu_g+\int_MR_gw^2d\mu_g.
\]
Thus \(w=u>0\) and \(R_g=0\). Moreover, \eqref{eq:conformal-scalar-curvature} and \eqref{eq:conformal-mass} give
\[
R_{\bar g}=0,\qquad m_E(\bar g)=m_E(g)-2\mathfrak c_E(g)=0.
\]
Applying \eqref{eq:conformal-scalar-curvature} to \(g=V^{4/(n-2)}\bar g\) then gives \(\Delta_{\bar g}V=0\).

To prove that \(\bar g\) is Ricci flat, let \(h\) be a smooth compactly supported symmetric \((0,2)\)-tensor and set \(g_s:=g+sh\). For small \(|s|\), let \(u_s\) minimize \(Q_{g_s}\) on \(\mathcal A_E\). Then
\[
(-c_n\Delta_{g_s}+R_{g_s})(u_s-u)=c_n(\Delta_{g_s}-\Delta_g)u-R_{g_s}u.
\]
The right-hand side is supported in \(K:=\operatorname{supp}h\) and bounded by \(C|s|\). Multiplying the equation by \(u_s-u\), Lemma~\ref{lem:local-Poincare} and \eqref{eq:perturbative-coercivity} give
\[
    c\|u_s-u\|_{\mathcal D_0(g)}^2\le Q_{g_s}(u_s-u)\le C|s|\|u_s-u\|_{\mathcal D_0(g)}.
\]
Hence \(\|u_s-u\|_{\mathcal D_0(g)}=O(|s|)\) and \(Q_{g_s}(u_s-u)=O(s^2)\). By the Euler-Lagrange equation for \(u_s\),
\[
Q_{g_s}(u_s)=Q_{g_s}(u)-Q_{g_s}(u_s-u)=Q_{g_s}(u)+O(s^2).
\]
Since \(g_s=g\) outside \(K\), \(m_E(g_s)=m_E(g)\). Lemma~\ref{lem:perturbative-conformal-mass-bound} gives
\[
Q_{g_s}(u_s)\le 2(n-1)|\mathbb S^{n-1}|m_E(g_s)=2(n-1)|\mathbb S^{n-1}|m_E(g)=Q_g(u).
\]
Together these imply
\[
\frac{d}{ds}\bigg|_{s=0}Q_{g_s}(u)=0.
\]
Set \(\hat g_s:=u^{4/(n-2)}g_s\). By \eqref{eq:conformal-scalar-curvature},
\[
R_{\hat g_s}d\mu_{\hat g_s}=u(-c_n\Delta_{g_s}u+R_{g_s}u)d\mu_{g_s},\qquad R_{\bar g}d\mu_{\bar g}=u(-c_n\Delta_gu+R_gu)d\mu_g.
\]
Subtracting these identities and integrating by parts, the boundary terms cancel because \(g_s=g\) outside \(K\). Hence
\[
\begin{aligned}
    Q_{g_s}(u)-Q_g(u)=&\int_Mu(-c_n\Delta_{g_s}u+R_{g_s}u)d\mu_{g_s}-\int u(-c_n\Delta_gu+R_gu)d\mu_g\\
    =&\int_MR_{\hat g_s}d\mu_{\hat g_s}-\int_MR_{\bar g}d\mu_{\bar g}=\int_MR_{\hat g_s}d\mu_{\hat g_s}.
\end{aligned}
\]
By the first variation of total scalar curvature (Einstein-Hilbert action),
\[
0=\frac{d}{ds}\bigg|_{s=0}Q_{g_s}(u)=\frac{d}{ds}\bigg|_{s=0}\int_MR_{\hat g_s}d\mu_{\hat g_s}=-\int_M\left\langle\operatorname{Ric}_{\bar g},u^{\frac{4}{n-2}}h\right\rangle_{\bar g} d\mu_{\bar g}.
\]
Since \(h\) is arbitrary and \(u>0\), it follows that \(\operatorname{Ric}_{\bar g}=0\).
\end{proof}

\section{Approximation by complete conformal metrics}

\subsection{The approximating metrics}

For the remainder of the paper, we assume equality in the mass--capacity inequality. Proposition~\ref{prop:equality-reduction} gives
\[
\operatorname{Ric}_{\bar g}=0,\qquad m_E(\bar g)=0,\qquad\Delta_{\bar g}V=0,
\]
where \(V=u^{-1}\). Moreover, \(g=V^{4/(n-2)}\bar g\) is complete, although \(\bar g\) need not be complete.

For \(0<t<1\), set
\[
\mathfrak c_t:=t\mathfrak c_E,\qquad V_t:=(1-t)+tV,\qquad g_t:=V_t^{\frac{4}{n-2}}\bar g,\qquad F_t:=2\log V_t.
\]
The function \(V_t\) is \(\bar g\)-harmonic, and \(g_t\to\bar g\) locally smoothly as \(t\to 0\). Moreover,
\[
g_t\ge (tV)^{\frac{4}{n-2}}\bar g= t^{\frac{4}{n-2}}g,
\]
so \(g_t\) is complete. Equation~\eqref{eq:conformal-mass} gives
\begin{equation}\label{eq:mass-gt}
    m_E(g_t)=2\mathfrak c_t=2t\mathfrak c_E.
\end{equation}

For a Riemannian manifold \((M^n,h)\), \(f\in C^\infty(M)\), and \(N\in\mathbb R\setminus\{n\}\), the \(N\)-Bakry--\'Emery Ricci tensor is defined by
\[
\operatorname{Ric}_f^N(h):=\operatorname{Ric}_h+\operatorname{Hess}_hf-\frac{1}{N-n}df\otimes df.
\]

We next record the conformal transformation formula needed below; see \cite[Proposition~4.4]{Case} for a general formula.
\begin{proposition}\label{prop:conformal-weighted-ricci}
Let \((M^n,h)\), \(n\ge 3\), be a Riemannian manifold and let \(v>0\). Set \(\hat h:=v^{4/(n-2)}h\) and \(f:=2\log v\). Then
\[
\operatorname{Ric}_f^{4-n}(\hat h)=\operatorname{Ric}_h-\frac{2v^{-1}\Delta_hv}{n-2}h+\frac{2}{n-2}\left(|\nabla\log v|_h^2h-d\log v\otimes d\log v\right).
\]
In particular, if \(\operatorname{Ric}_h\ge 0\) and \(\Delta_h v=0\), then 
\[
\operatorname{Ric}_f^{4-n}(\hat h)\ge\frac{2}{n-2}\left(|\nabla\log v|_h^2h-d\log v\otimes d\log v\right)\ge 0.
\]
\end{proposition}
\begin{proof}
    Let \(\phi:=2\log v/(n-2)\), so that \(\hat h=e^{2\phi}h\) and \(f=(n-2)\phi\). Then
    \[
    \operatorname{Ric}_{\hat h}=\operatorname{Ric}_h-(n-2)\left(\operatorname{Hess}_h\phi-d\phi\otimes d\phi\right)-\left(\Delta_h\phi+(n-2)|\nabla\phi|_h^2\right)h
    \]
    and
    \[
    \begin{aligned}
        \operatorname{Hess}_{\hat h}f=&\operatorname{Hess}_hf-df\otimes d\phi-d\phi\otimes df+\langle\nabla f,\nabla\phi\rangle_hh\\
        =&(n-2)\operatorname{Hess}_h\phi-2(n-2)d\phi\otimes d\phi+(n-2)|\nabla\phi|_h^2h,
    \end{aligned}
    \]
    and
    \[
    \frac{1}{2(n-2)}df\otimes df=\frac{n-2}{2}d\phi\otimes d\phi.
    \]
    Adding together gives
    \[
    \operatorname{Ric}_f^{4-n}(\hat h)=\operatorname{Ric}_h-h\Delta_h\phi-\frac{n-2}{2}d\phi\otimes d\phi.
    \]
    Since 
    \[
    \Delta_h\phi=\frac{2}{n-2}\left(v^{-1}\Delta_hv-|\nabla\log v|_h^2\right),\qquad d\phi=\frac{2}{n-2}d\log v,
    \]
    we obtain
    \[
    \operatorname{Ric}_f^{4-n}(\hat h)=\operatorname{Ric}_h-\frac{2v^{-1}\Delta_hv}{n-2}h+\frac{2}{n-2}\left(|\nabla\log v|_h^2h-d\log v\otimes d\log v\right).
    \]
\end{proof}

Since \(\operatorname{Ric}_{\bar g}=0\) and \(\Delta_{\bar g}V_t=0\), Proposition~\ref{prop:conformal-weighted-ricci} applied with \(h=\bar g\), \(v=V_t\) and \(f=F_t\) gives
\begin{equation}\label{eq:weighted-ricci-nonnegative}
\operatorname{Ric}_{F_t}^{4-n}(g_t)=\frac{1}{2(n-2)}\left(|\bar\nabla F_t|_{\bar g}^2\bar g-dF_t\otimes dF_t\right)\ge 0.
\end{equation}

\subsection{Weighted harmonic coordinates}

In asymptotically flat coordinates \(x\) on \(E\), set \(r:=|x|\) and \(\rho:=\max\{1,r\}\). Fix \(p>n\). For \(k\in\mathbb N_{\ge 0}\) and \(\delta\in\mathbb R\), let \(W_{\delta}^{k,p}(E)\) denote the weighted Sobolev space with norm
\[
\|\varphi\|_{W_\delta^{k,p}(E)}:=\sum_{j=0}^k\left(\int_E|\nabla^j\varphi|^p\rho^{-(\delta-j)p-n}dx\right)^{\frac{1}{p}}.
\]
This space is independent of the choice of asymptotically flat coordinates.

\begin{lemma}\label{lem:harmonic-asymptotics}
After replacing \(E\) by a smaller exterior region, there exist \(R_1>0\), \(q>(n-2)/2\),  \(\sigma>0\) and coordinates
\[
x=(x^1,\ldots,x^n):E\longrightarrow\mathbb R^n\setminus\overline{B_{R_1}}
\]
such that \(\Delta_{\bar g}x^i=0\) for \(i=1,\ldots,n\),
\begin{equation}\label{eq:harmonic-asymptotics}
\bar g_{ij}-\delta_{ij}\in W_{-q}^{2,p}(E),\qquad V=1+\mathfrak c_Er^{2-n}+O_2(r^{2-n-\sigma}),
\end{equation}
where \(\mathfrak c_E:=\mathfrak c_E(g)\). Moreover, \(\bar g_{ij}-\delta_{ij}=O_2(r^{-q})\).
\end{lemma}
\begin{proof}
By definition, there exist \(R_0>0\) and coordinates
\[
x_0=(x_0^1,\ldots,x_0^n):E\to\mathbb R^n\setminus\overline{B_{R_0}}
\]
in which \(g_{ij}-\delta_{ij}=O_2(r_0^{-\tau})\), where \(r_0:=|x_0|\). Applying \cite[Lemma~2.1]{BZ}, we obtain \(u=1-ar_0^{2-n}+O_2(r_0^{2-n-\sigma})\) for some \(a\in\mathbb R\), \(\sigma>0\). As in the proof of Proposition~\ref{prop:conformal-mass-bound}, \(a=\mathfrak c_E\). 
Set \(\tau_0:=\min\{\tau, n-2\}\). Note that \(u^{4/(n-2)}-1=O_2(r_0^{2-n})\) and \(g_{ij}-\delta_{ij}=O_2(r_0^{-\tau_0})\), we have
\[
\bar g_{ij}-\delta_{ij}=O_2(r_0^{-\tau_0}).
\]
Choose \((n-2)/2<q<\tau_0\). Thus, in the \(x_0\)-coordinates, \(\bar g_{ij}-\delta_{ij}\in W_{-q}^{2,p}(E)\). 

Extend \(\bar g\) from \(\mathbb R^n\setminus\overline{B_{R_0}}\) to a smooth metric on \(\mathbb R^n\). By \cite[Proposition~3.3]{Bartnik} and shrinking \(E\) if necessary, we obtain coordinates
\[
x:E\longrightarrow\mathbb R^n\setminus\overline{B_{R_1}}
\]
such that \(\Delta_{\bar g}x^i=0\) and in which \(\bar g_{ij}-\delta_{ij}\in W_{-q}^{2,p}(E)\). Moreover, after a rotation, we may assume that \(x-x_0=o(r_0)\). Since \(x\) is harmonic, the equation \(\operatorname{Ric}_{\bar g}=0\) takes the form
\[
\bar g^{ab}\partial_a\partial_b\bar g_{ij}=\bar g^{-1}*\bar g^{-1}*\partial\bar g*\partial\bar g.
\]
Standard elliptic estimates give \(\bar g_{ij}-\delta_{ij}=O_2(r^{-q})\). Finally, applying \cite[Lemma~2.1]{BZ}  as before gives \(u=1-\mathfrak c_Er^{2-n}+O_2(r^{2-n-\sigma})\). Thus
\[
V=u^{-1}=1+\mathfrak c_Er^{2-n}+O_2(r^{2-n-\sigma}).
\]
\end{proof}

For a Riemannian metric \(h\) and smooth functions \(H,U\), write
\[
\Delta_{h,H}U:=\Delta_hU-\langle\nabla_hH,\nabla_hU\rangle_h.
\]
Set \(T_t:=F_t/(n-1)=2\log V_t/(n-1)\) and 
\[
\mathcal L_tU:=\Delta_{\bar g,T_t}U=\Delta_{\bar g}U-\langle\bar\nabla T_t,\bar\nabla U\rangle_{\bar g}.
\]
Lemma~\ref{lem:harmonic-asymptotics} gives, uniformly for \(0<t<1\),
\begin{equation}\label{eq:Tt-uniform-asymptotic}
    T_t=\frac{2\mathfrak c_t}{n-1}r^{2-n}+O_2(tr^{2-n-\sigma}).
\end{equation}
Set \(d\mathfrak m_t:=e^{-F_t}d\mu_{g_t}\). Since \(g_t=V_t^{4/(n-2)}\bar g\) and \(F_t=2\log V_t\),
\begin{equation}\label{eq:measure-cancellation}
d\mathfrak m_t=V_t^{\frac{4}{n-2}}d\mu_{\bar g},\qquad |\nabla_{g_t}U|_{g_t}^2d\mathfrak m_t=|\bar\nabla U|_{\bar g}^2d\mu_{\bar g}.
\end{equation}
Consequently, 
\[
\Delta_{g_t,F_t}U=V_t^{-\frac{4}{n-2}}\Delta_{\bar g}U,\qquad \Delta_{g_t,F_t+T_t}U=V_t^{-\frac{4}{n-2}}\mathcal L_tU.
\]
Choose \(\delta\in(2-n,3-n)\) such that \(\delta-2>1-n-\sigma\). For \(R\ge R_1\), set \(E_R:=\{p\in E:r(p)>R\}\).

\begin{proposition}\label{prop:drift-coordinates}
Assume \(n\ge 4 \). There exist \(R,C>0\) such that, for every \(0<t<1\), there are smooth functions \(X_t^1,\ldots,X_t^n\) on \(M\) satisfying
\begin{equation}\label{eq:X-drift}
\Delta_{g_t,F_t+T_t}X_t^i=0.
\end{equation}
On \(E_R\),
\begin{equation}\label{eq:X-asymptotic}
\left\|X_t^i-x^i-\frac{\mathfrak c_t}{n-1}x^ir^{2-n}\right\|_{W_\delta^{3,p}(E_R)}\le C.
\end{equation}
Moreover, \(X_t|_{E_R}\) is a coordinate chart, and
\begin{equation}\label{eq:finite-energy-arbitrary-ends}
\int_{M\setminus E_{2R}}|\nabla_{g_t}X_t^i|_{g_t}^2d\mathfrak m_t<\infty.
\end{equation}
\end{proposition}

\begin{proof}
\noindent\emph{Step 1.} Let \(R_2>R_1\) to be fixed. Choose smooth extensions \(\tilde g\) and \(\widetilde T_t\) to \(\mathbb R^n\) such that
\[
(\tilde g,\widetilde T_t)=(g_{\mathrm{Euc}},0)\quad\text{on }B_{R_2},\qquad (\tilde g,\widetilde T_t)=(\bar g, T_t)\quad\text{on }\mathbb R^n\setminus B_{2R_2},
\]
and set 
\[
\widetilde{\mathcal L}_t U:=\Delta_{\tilde g}U-\left\langle\nabla_{\tilde g}\widetilde T_t,\nabla_{\tilde g}U\right\rangle_{\tilde g}.
\]
By \eqref{eq:harmonic-asymptotics} and \eqref{eq:Tt-uniform-asymptotic}, 
\[
\|(\widetilde{\mathcal L}_t-\Delta_{\mathbb R^n})Z\|_{W_{\delta-2}^{k,p}(\mathbb R^n)}\le \varepsilon(R_2)\|Z\|_{W_\delta^{k+2,p}(\mathbb R^n)}\quad (k=0,1),\qquad \|\nabla_{\tilde g}\widetilde T_t\|_{L^n(\mathbb R^n,\tilde g)}\le\varepsilon(R_2),
\]
where \(\varepsilon(R_2)\to 0\) as \(R_2\to\infty\), uniformly for \(0<t<1\).

Since \(\delta\in (2-n,0)\), \cite[Propositions~1.6 and~2.2]{Bartnik} gives
\[
\Delta_{\mathbb R^n}: W_\delta^{k+2,p}(\mathbb R^n)\longrightarrow W_{\delta-2}^{k,p}(\mathbb R^n)
\]
as an isomorphism for \(k=0,1\). For \(R_2\) sufficiently large, \(\widetilde{\mathcal L}_t\) is therefore an isomorphism, with
\begin{equation}\label{eq:global-weighted-inverse}
\|Z\|_{W_\delta^{k+2,p}(\mathbb R^n)}\le C\|\widetilde{\mathcal L}_tZ\|_{W_{\delta-2}^{k,p}(\mathbb R^n)}\qquad (k=0,1).
\end{equation}
Moreover, after increasing  \(R_2\) if necessary, for every \(U\in\mathcal D_0(\tilde g)\),
\begin{equation}\label{eq:extended-coercivity}
\int_{\mathbb R^n}\left\langle\nabla_{\tilde g}U,\nabla_{\tilde g}U+U\nabla_{\tilde g}\widetilde T_t\right\rangle_{\tilde g}d\mu_{\tilde g}\ge\frac{1}{2}\|U\|_{\mathcal D_0(\tilde g)}^2.
\end{equation}
Estimate \eqref{eq:global-weighted-inverse} also holds for solutions in \(\mathcal D_0(\tilde g)\). Let \(U\in\mathcal D_0(\tilde g)\) and \(f\in W_{\delta-2}^{1,p}(\mathbb R^n)\) satisfy \(\widetilde{\mathcal L}_tU=f\), and set
\[
Z_f:=(\widetilde{\mathcal L}_t)^{-1}f\in W_\delta^{3,p}(\mathbb R^n).
\]
Since \(n\ge 4\) and \(\delta<3-n\), we have \(\delta<(2-n)/2\). The weighted Sobolev inequality \cite[Theorem~1.2(iv)]{Bartnik} gives
\[
|Z_f|+r|\nabla Z_f|\le Cr^\delta\|f\|_{W_{\delta-2}^{1,p}(\mathbb R^n)}.
\]
Let \(\eta_s=1\) on \(B_s\), \(\eta_s=0\) on \(\mathbb R^n\setminus B_{2s}\) and \(|D\eta_s|\le C/s\). Then
\[
\begin{aligned}
   \|Z_f-\eta_sZ_f\|_{\mathcal D_0(\tilde g)}^2&\le C\int_{r>s}|\nabla Z_f|^2dx+\frac{C}{s^2}\int_{s<r<2s}|Z_f|^2dx\\
&\le Cs^{2\delta+n-2}\|f\|_{W_{\delta-2}^{1,p}(\mathbb R^n)}^2. 
\end{aligned}
\]
Thus \(Z_f\in\mathcal D_0(\tilde g)\). Since \(\widetilde{\mathcal L}_t(U-Z_f)=0\), \eqref{eq:extended-coercivity} gives \(U=Z_f\). Consequently,
\begin{equation}\label{eq:D0-weighted-estimate}
U\in W_\delta^{3,p}(\mathbb R^n),\qquad \|U\|_{W_\delta^{3,p}(\mathbb R^n)}\le C\|f\|_{W_{\delta-2}^{1,p}(\mathbb R^n)}.
\end{equation}

\noindent\emph{Step 2.} Since \(\Delta_{\bar g}x^i=0\), we have \(\mathcal L_tx^i=-\bar g^{ai}\partial_aT_t\). Hence
\[
\mathcal L_tx^i=\frac{2(n-2)}{n-1}\mathfrak c_tx^ir^{-n}+O_1(tr^{1-n-\sigma}).
\]
On the other hand, \(\Delta_{\mathbb R^n}(x^ir^{2-n})=-2(n-2)x^ir^{-n}\) gives
\[
\mathcal L_t(x^ir^{2-n})=-2(n-2)x^ir^{-n}+O_1(r^{1-n-\sigma}).
\]
Thus \(\delta-2>1-n-\sigma\) implies
\[
\left\|\mathcal L_t\left(x^i+\frac{\mathfrak c_t}{n-1}x^ir^{2-n}\right)\right\|_{W_{\delta-2}^{1,p}(E_{2R_2})}\le Ct.
\]
Extend this function to \(\tilde f_t^i\in W_{\delta-2}^{1,p}(\mathbb R^n)\) with \(\|\tilde f_t^i\|_{W_{\delta-2}^{1,p}(\mathbb R^n)}\le Ct\). By \eqref{eq:global-weighted-inverse}, there exists \(Z_t^i\in W_\delta^{3,p}(\mathbb R^n)\) satisfying
\[
\widetilde{\mathcal L}_tZ_t^i=-\tilde f_t^i,\qquad \|Z_t^i\|_{W_\delta^{3,p}(\mathbb R^n)}\le Ct.
\]
Hence
\[
Y_t^i:=x^i+\frac{\mathfrak c_t}{n-1}x^ir^{2-n}+Z_t^i
\]
is \(\mathcal L_t\)-harmonic on \(E_{2R_2}\) and 
\begin{equation}\label{eq:Yt-asymptotic}
\left\|Y_t^i-x^i-\frac{\mathfrak c_t}{n-1}x^ir^{2-n}\right\|_{W_\delta^{3,p}(E_{2R_2})}\le Ct.
\end{equation}
\noindent\emph{Step 3.} Since \(V_t\) is \(\bar g\)-harmonic and \(T_t=2\log V_t/(n-1)\),
\[
\Delta_{\bar g}T_t=-\frac{n-1}{2}|\bar\nabla T_t|_{\bar g}^2.
\]
For \(\psi\in C_c^\infty(M)\), multiplying by \(\psi^2\) and integrating by parts gives
\[
\|\psi\bar\nabla T_t\|_{L^2(M,\bar g)}\le \frac{4}{n-1}\|\bar\nabla\psi\|_{L^2(M,\bar g)}.
\]
For \(U,\varphi\in C_c^\infty(M)\), set
\[
\mathfrak b_t(U,\varphi):=\int_M\left\langle\bar\nabla U,\bar\nabla\varphi+\varphi\bar\nabla T_t\right\rangle_{\bar g}d\mu_{\bar g}.
\]
Then 
\[
\begin{aligned}
    |\mathfrak b_t(U,\varphi)|&\le\|\bar\nabla U\|_{L^2(M,\bar g)}\left(\|\bar\nabla\varphi\|_{L^2(M,\bar g)}+\|\varphi\bar\nabla T_t\|_{L^2(M,\bar g)}\right)\\
    &\le\left(1+\frac{4}{n-1}\right)\|U\|_{\mathcal D_0(\bar g)}\|\varphi\|_{\mathcal D_0(\bar g)},
\end{aligned}
\]
while
\[
\begin{aligned}
   \mathfrak b_t(U,U)&=\|U\|_{\mathcal D_0(\bar g)}^2-\frac{1}{2}\int_MU^2\Delta_{\bar g}T_td\mu_{\bar g}\\
&=\|U\|_{\mathcal D_0(\bar g)}^2+\frac{n-1}{4}\int_MU^2|\bar\nabla T_t|_{\bar g}^2d\mu_{\bar g}\ge\|U\|_{\mathcal D_0(\bar g)}^2. 
\end{aligned}
\]
Choose \(\beta\in C^\infty(M)\) with \(\operatorname{supp}\beta\subset E_{2R_2}\) and \(\beta=1\) on \(E_{3R_2}\). Then \(\mathcal L_t(\beta Y_t^i)\) is supported in \(\overline{E_{2R_2}\setminus E_{3R_2}}\) and \eqref{eq:Yt-asymptotic} gives
\[
\|\mathcal L_t(\beta Y_t^i)\|_{W^{1,p}(M,\bar g)}\le C\|Y_t^i\|_{W^{2,p}(E_{2R_2}\setminus E_{3R_2},\bar g)}\le C.
\]

 By Lemma~\ref{lem:local-Poincare},
\[
\left|\int_M\varphi\mathcal L_t(\beta Y_t^i)d\mu_{\bar g}\right|\le C\|\varphi\|_{L^{\frac{2n}{n-2}}(E_{2R_2}\setminus E_{3R_2},\bar g)}\le C\|\varphi\|_{\mathcal D_0(\bar g)}.
\]
The Lax--Milgram theorem therefore gives \(\xi_t^i\in\mathcal D_0(\bar g)\) satisfying
\[
\mathfrak b_t(\xi_t^i,\varphi)=\int_M\varphi\mathcal L_t(\beta Y_t^i)d\mu_{\bar g}\qquad (\varphi\in\mathcal D_0(\bar g)),
\]
with \(\|\xi_t^i\|_{\mathcal D_0(\bar g)}\le C\). Since \(\mathfrak b_t(U,\varphi)=-\int_M\varphi\mathcal L_tU d\mu_{\bar g}\), we know
\[
\mathcal L_t\xi_t^i=-\mathcal L_t(\beta Y_t^i).
\]
Interior elliptic estimates on \(E_{2R_2}\setminus E_{3R_2}\), together with Lemma~\ref{lem:local-Poincare}, give
\[
\|\xi_t^i\|_{W^{2,p}(E_{2R_2}\setminus E_{3R_2},\bar g)}\le C\left(\|\xi_t^i\|_{\mathcal D_0(\bar g)}+\|\mathcal L_t(\beta Y_t^i)\|_{W^{1,p}(M,\bar g)}\right)\le C.
\]
Extending \(\beta\xi_t^i\) by zero to \(\mathbb R^n\), we obtain \(\widetilde{\mathcal L}_t(\beta\xi_t^i)=[\mathcal L_t,\beta]\xi_t^i-\beta\mathcal L_t(\beta Y_t^i)\), and hence
\[
\|\widetilde{\mathcal L}_t(\beta\xi_t^i)\|_{W_{\delta-2}^{1,p}(\mathbb R^n)}\le C\|\xi_t^i\|_{W^{2,p}(E_{2R_2}\setminus E_{3R_2},\bar g)}+C\|\mathcal L_t(\beta Y_t^i)\|_{W^{1,p}(M,\bar g)}\le C.
\]
Since \(\beta\xi_t^i\in\mathcal D_0(\tilde g)\) and \(\beta=1\) on \(E_{3R_2}\), \eqref{eq:D0-weighted-estimate} gives
\[
\|\xi_t^i\|_{W_\delta^{3,p}(E_{3R_2})}\le\|\beta\xi_t^i\|_{W_\delta^{3,p}(\mathbb R^n)}\le C.
\]

Set \(X_t^i:=\beta Y_t^i+\xi_t^i\). Since \(\mathcal L_t\xi_t^i=-\mathcal L_t(\beta Y_t^i)\), we have \(\mathcal L_tX_t^i=0\) on \(M\) and hence \eqref{eq:X-drift}. Note that \(X_t^i=Y_t^i+\xi_t^i\) on \(E_{3R_2}\). Thus \eqref{eq:Yt-asymptotic} and \(\|\xi_t^i\|_{W_\delta^{3,p}(E_{3R_2})}\le C\) prove \eqref{eq:X-asymptotic} for every \(R\ge 3R_2\).

Moreover, for \(R\ge R_2\), 
\[
\int_{M\setminus E_{2R}}|\bar\nabla X_t^i|_{\bar g}^2d\mu_{\bar g}\le 2\int_{M\setminus E_{2R}}|\bar\nabla(\beta Y_t^i)|_{\bar g}^2d\mu_{\bar g}+2\|\xi_t^i\|_{\mathcal D_0(\bar g)}^2<\infty.
\]
Together with \eqref{eq:measure-cancellation}, this proves \eqref{eq:finite-energy-arbitrary-ends}.

\noindent\emph{Step 4.} Write \(X_t=x+\zeta_t\) on \(E_{3R_2}\). By \eqref{eq:X-asymptotic} and  weighted Sobolev inequality \cite[Theorem~1.2(iv)]{Bartnik},
\[
|\zeta_t|+r|\nabla\zeta_t|\le C\left(r^{3-n}+r^\delta\right)\le Cr^{3-n},\qquad \delta<3-n.
\]
Choose \(R\) sufficiently large that \(E_{R/2}\subset E_{3R_2}\) and
\[
\sup_{E_{R/2}}|\nabla\zeta _t|<\frac{1}{2},\qquad \sup_{E_R}|\zeta_t|<\frac{R}{4}.
\]
Then \(X_t|_{E_R}\) is a coordinate chart.
\end{proof}

\section{Rigidity}

\subsection{The weighted Bochner argument}

Assume \(n\ge 4\). For simplicity, throughout this subsection, \(\nabla,\Delta,\operatorname{Hess},\operatorname{div}\), inner products and norms are taken with respect to \(g_t\). For \(H\in C^\infty(M)\) and \(Y\in\Gamma(TM)\), write \(\operatorname{div}_HY:=e^H\operatorname{div}(e^{-H}Y)\). We abbreviate \(\operatorname{Ric}_{F_t}^{4-n}(g_t)\) to \(\operatorname{Ric}_{F_t}^{4-n}\). For \(X\in C^\infty(M)\), write
\[
\operatorname{Hess}^{\circ}X:=\operatorname{Hess}X-\frac{1}{n}(\Delta X)g_t,\qquad \Gamma_n:=\frac{(n-1)(n-3)}{2(n-2)}>0.
\]

\begin{lemma}\label{lem:weighted-bochner}
If \(X\in C^\infty(M)\) satisfies \(\Delta_{F_t+T_t}X=0\), then
\begin{equation}\label{eq:modified-bochner}
\begin{aligned}
&\operatorname{div}_{F_t}\left(\frac{1}{2}\nabla|\nabla X|^2-(\Delta_{F_t}X)\nabla X\right)\\
=&\frac{1}{2}\Delta_{F_t}|\nabla X|^2-\operatorname{div}_{F_t}\left((\Delta_{F_t}X)\nabla X\right)\\
=&|\operatorname{Hess}^{\circ}X|^2+\operatorname{Ric}_{F_t}^{4-n}(\nabla X,\nabla X)+\Gamma_n\langle\nabla T_t,\nabla X\rangle^2=:e_t(X).
\end{aligned}
\end{equation}
\end{lemma}
\begin{proof}
Since \(F_t+T_t=nT_t\), equation \(\Delta_{F_t+T_t}X=0\) gives
\[
\Delta X=n\langle\nabla T_t,\nabla X\rangle,\qquad\langle\nabla F_t,\nabla X\rangle=(n-1)\Delta_{F_t}X.
\]
By the weighted Bochner formula,
\begin{equation}\label{eq:weighted-bochner-identity}
\frac{1}{2}\Delta_{F_t}|\nabla X|^2=|\operatorname{Hess}X|^2+\langle\nabla X,\nabla\Delta_{F_t}X\rangle+(\operatorname{Ric}_{g_t}+\operatorname{Hess}F_t)(\nabla X,\nabla X).
\end{equation}
Moreover, from \(\operatorname{div}_{F_t}Y=e^{F_t}\operatorname{div}(e^{-F_t}Y)\),
\begin{equation}\label{eq:weighted-divergence-identity}
\operatorname{div}_{F_t}\left((\Delta_{F_t}X)\nabla X\right)=\langle\nabla\Delta_{F_t}X,\nabla X\rangle+(\Delta_{F_t}X)^2.
\end{equation}
 Since \(\Delta X=n\langle\nabla T_t,\nabla X\rangle\), the trace decomposition becomes
\begin{equation}\label{eq:hessian-trace-decomposition}
|\operatorname{Hess}X|^2=|\operatorname{Hess}^{\circ}X|^2+n\langle\nabla T_t,\nabla X\rangle^2.
\end{equation}
Moreover, by definition
\[
\operatorname{Ric}_{g_t}+\operatorname{Hess}F_t=\operatorname{Ric}_{F_t}^{4-n}-\frac{1}{2(n-2)}dF_t\otimes dF_t.
\]
Since \(\langle\nabla F_t,\nabla X\rangle=(n-1)\Delta_{F_t}X\), we also have
\[
(dF_t\otimes dF_t)(\nabla X,\nabla X)=(n-1)^2(\Delta_{F_t}X)^2.
\]

Subtracting \eqref{eq:weighted-divergence-identity} from \eqref{eq:weighted-bochner-identity} and using identities above together with \(dF_t=(n-1)dT_t\), we obtain
\[
\begin{aligned}
    &\frac{1}{2}\Delta_{F_t}|\nabla X|^2-\operatorname{div}_{F_t}\left(\langle\nabla T_t,\nabla X\rangle\nabla X\right)\\
    =&|\operatorname{Hess}^{\circ} X|^2+\operatorname{Ric}_{F_t}^{4-n}(\nabla X,\nabla X)+\left(n-1-\frac{(n-1)^2}{2(n-2)}\right)\langle\nabla T_t,\nabla X\rangle^2\\
    =&|\operatorname{Hess}^{\circ}X|^2+\operatorname{Ric}_{F_t}^{4-n}(\nabla X,\nabla X)+\Gamma_n\langle\nabla T_t,\nabla X\rangle^2.
\end{aligned}
\]
\end{proof}
Note that by \eqref{eq:weighted-ricci-nonnegative} and \eqref{eq:hessian-trace-decomposition},
\begin{equation}\label{eq:pointwise-hessian-control}
|\operatorname{Hess}X|^2\le\max\left\{1,\frac{n}{\Gamma_n}\right\}e_t(X).
\end{equation}
\begin{remark}
If \(X\) were instead \(\Delta_{F_t}\)-harmonic, the same calculation would give
\[
\frac{1}{2}\Delta_{F_t}|\nabla X|^2=|\operatorname{Hess}^{\circ}X|^2+\operatorname{Ric}_{F_t}^{4-n}(\nabla X,\nabla X)+\frac{n-4}{2n(n-2)}\langle\nabla F_t,\nabla X\rangle^2.
\]
The last coefficient is positive for \(n\ge 5\) but vanishes for \(n=4\). This fits the heuristic that, in the range \(1/N\le 1/n\), geometry improves as \(1/N\) increases.
The drift \(T_t=F_t/(n-1)\) allows us to treat dimensions \(n\ge 4\) uniformly. 
\end{remark}

\begin{lemma}\label{lem:coordinate-boundary-integral}
Let \((M^n,h,E)\) be asymptotically flat with arbitrary ends and let \(X=(X^1,\ldots,X^n)\) be an asymptotically flat coordinate system on \(E\), and set \(r:=|X|\). Suppose that \(F\in C^\infty(M)\) satisfies
\[
F=br^{2-n}+O_2(r^{2-n-\sigma})
\]
for some \(b\in\mathbb R\) and \(\sigma>0\). For \(S_R:=\{|X|=R\}\),
\begin{equation}\label{eq:coordinate-boundary-integral}
    \begin{aligned}
        &\lim_{R\to\infty}\sum_{\alpha=1}^n\int_{S_R}e^{-F}\left[\frac{1}{2}\partial_\nu|\nabla_hX^\alpha|_h^2-(\Delta_{h,F}X^\alpha)\partial_\nu X^\alpha\right]d\sigma_h\\
        \qquad&=2(n-1)|\mathbb S^{n-1}|m_E(h)-(n-2)|\mathbb S^{n-1}|b,
    \end{aligned}
\end{equation}
where \(\nu\) is the \(h\)-unit normal pointing toward infinity.
\end{lemma}
\begin{proof}
    Write \(h_{ij}=\delta_{ij}+p_{ij}\). Since \(X\) is an asymptotically flat coordinate system, \(p_{ij}=O_2(r^{-q})\) for some \(q>(n-2)/2\). Let \(\nu_0\) and \(d\sigma_0\) denote the Euclidean unit normal and area measure in the \(X\)-coordinates. Then
    \[
    h^{ij}=\delta^{ij}-p_{ij}+O_2(r^{-2q}),\qquad\nu=\nu_0+O(r^{-q}),\qquad d\sigma_h=(1+O(r^{-q}))d\sigma_0.
    \]
    Noting that \(\Gamma_{ij}^\alpha=\frac{1}{2}\left(\partial_ip_{\alpha j}+\partial_jp_{i\alpha}-\partial_\alpha p_{ij}\right)+O(r^{-2q-1})\), we have
    \[
    \Delta_hX^\alpha=-h^{ij}\Gamma_{ij}^\alpha=-\partial_jp_{j\alpha}+\frac{1}{2}\partial_\alpha p_{jj}+O(r^{-2q-1}).
    \]
  Combining with \(|\nabla_hX^\alpha|_h^2=h^{\alpha\alpha}\), \(\partial_\nu X^\alpha=\nu_0^\alpha+O(r^{-q})\) and \(|dF|=O(r^{1-n})\), we obtain
  \[
  \begin{aligned}
      &\sum_{\alpha=1}^n\frac{1}{2}\partial_\nu|\nabla_hX^\alpha|_h^2=-\frac{1}{2}\partial_ip_{jj}\nu_0^i+O(r^{-2q-1}),\\
      &-\sum_{\alpha=1}^n(\Delta_hX^\alpha)\partial_\nu X^\alpha=\left(\partial_jp_{ji}-\frac{1}{2}\partial_ip_{jj}\right)\nu_0^i+O(r^{-2q-1}),\\
      &\sum_{\alpha=1}^n\langle\nabla_hF,\nabla_hX^\alpha\rangle_h\partial_\nu X^\alpha=\partial_iF\nu_0^i+O(r^{1-n-q}).
  \end{aligned}
  \]
 Thus
  \[
  \begin{aligned}
  &\sum_{\alpha=1}^n\left[\frac{1}{2}\partial_\nu|\nabla_hX^\alpha|_h^2-(\Delta_{h,F}X^\alpha)\partial_\nu X^\alpha\right]\\
  =&(\partial_jp_{ij}-\partial_ip_{jj}+\partial_iF)\nu_0^i+O(r^{-2q-1})+O(r^{1-n-q}).
  \end{aligned}
  \]
  Moreover,
  \[
  e^{-F}d\sigma_h=\left(1+O(r^{-q})+O(r^{2-n})\right)d\sigma_0.
  \]
  Since \(2q+1>n-1\),
  \[
  \begin{aligned}
       &\lim_{R\to\infty}\sum_{\alpha=1}^n\int_{S_R}e^{-F}\left[\frac{1}{2}\partial_\nu|\nabla_hX^\alpha|_h^2-(\Delta_{h,F}X^\alpha)\partial_\nu X^\alpha\right]d\sigma_h\\
       \qquad&=\lim_{R\to\infty}\int_{S_R}(\partial_jp_{ij}-\partial_ip_{jj}+\partial_iF)\nu_0^id\sigma_0\\
       \qquad&=\lim_{R\to\infty}\left[\int_{S_R}(\partial_jp_{ij}-\partial_ip_{jj})\nu_0^id\sigma_0+\int_{S_R}\partial_{\nu_0}Fd\sigma_0\right]\\
       \qquad&=2(n-1)|\mathbb S^{n-1}|m_E(h)-(n-2)|\mathbb S^{n-1}|b.
  \end{aligned}
  \]
\end{proof}

\begin{proposition}\label{prop:mass-controls-hessian}
There exists \(C=C(n)>0\) such that, for every \(0<t<1\), the functions \(X_t^i\) from Proposition~\ref{prop:drift-coordinates} satisfy
\begin{equation}\label{eq:mass-controls-hessian}
\sum_{i=1}^n\int_M|\operatorname{Hess}_{g_t}X_t^i|_{g_t}^2d\mathfrak m_t\le Cm_E(g_t)=2Ct\mathfrak c_E.
\end{equation}
\end{proposition}
\begin{proof}
    By \eqref{eq:harmonic-asymptotics} and \eqref{eq:X-asymptotic}, \(X_t\) is an asymptotically flat coordinate system for \(g_t\). Moreover, \eqref{eq:Tt-uniform-asymptotic} and \eqref{eq:X-asymptotic} give
    \[
    F_t=2\mathfrak c_t|X_t|^{2-n}+O_2\left(|X_t|^{2-n-\sigma}\right).
    \]
Set \(S_s^t:=\{|X_t|=s\}\), and
\[
\mathcal B_{t,s}:=\sum_{i=1}^n\int_{S_s^t}e^{-F_t}\left[\frac{1}{2}\partial_\nu|\nabla X_t^i|^2-(\Delta_{F_t}X_t^i)\partial_\nu X_t^i\right]d\sigma.
\]
Lemma~\ref{lem:coordinate-boundary-integral}, applied with \(b=2\mathfrak c_t\), and \eqref{eq:mass-gt} give
\begin{equation}\label{eq:boundary-integral-value}
\lim_{s\to\infty}\mathcal B_{t,s}=2(n-1)|\mathbb S^{n-1}|m_E(g_t)-2(n-2)|\mathbb S^{n-1}|\mathfrak c_t=n|\mathbb S^{n-1}|m_E(g_t).
\end{equation}
Set \(M_s^t:=M\setminus\{|X_t|>s\}\). Fix \(\chi\in C^\infty([0,\infty),\mathbb R)\) such that \(\chi=1\) on \([0,1]\) and \(\chi=0\) on \([2,\infty)\). For \(L>0\), define \(\eta_L(p):=\chi(d_{g_t}(p,S_s^t)/L)\) for \(p\in M_s^t\). Since \(g_t\) is a complete metric, \(\eta_L\) has compact support on \(M_s^t\). Moreover,
\[
    \eta_L=1\quad\text{if }d_{g_t}(\cdot,S_s^t)\le L,\qquad |\nabla\eta_L|\le\frac{C}{L}\quad\text{a.e.}.
\]
From \eqref{eq:X-drift}, we know \(\Delta X_t^i=n\Delta_{F_t}X_t^i\). Hence
\[
\frac{1}{2}\nabla|\nabla X_t^i|^2-(\Delta_{F_t}X_t^i)\nabla X_t^i=\left(\operatorname{Hess}^{\circ}X_t^i(\nabla X_t^i,\cdot)\right)^\sharp.
\]
Multiplying \eqref{eq:modified-bochner} by \(\eta_L^2\) and integrating by parts therefore give
\[
\sum_{i=1}^n\int_{M_s^t}\eta_L^2e_t(X_t^i)d\mathfrak m_t=\mathcal B_{t,s}-2\sum_{i=1}^n\int_{M_s^t}\eta_L\operatorname{Hess}^\circ X_t^i(\nabla X_t^i,\nabla\eta_L)d\mathfrak m_t.
\]
Since \(e_t(X_t^i)\ge|\operatorname{Hess}^{\circ}X_t^i|^2\), Young's inequality yields
\[
\frac{1}{2}\sum_{i=1}^n\int_{M_s^t}\eta_L^2e_t(X_t^i)d\mathfrak m_t\le\mathcal B_{t,s}+2\|\nabla\eta_L\|_{L^\infty}^2\sum_{i=1}^n\int_{\operatorname{supp}d\eta_L}|\nabla X_t^i|^2d\mathfrak m_t.
\]
By \eqref{eq:finite-energy-arbitrary-ends}, \(\sum_{i=1}^n\int_{M_s^t}|\nabla X_t^i|^2d\mathfrak m_t<\infty\). Then letting first \(L\to \infty\) and then \(s\to\infty\), \eqref{eq:boundary-integral-value} gives
\[
\sum_{i=1}^n\int_Me_t(X_t^i)d\mathfrak m_t\le 2n|\mathbb S^{n-1}|m_E(g_t).
\]
The estimate \eqref{eq:pointwise-hessian-control} now gives the first inequality in \eqref{eq:mass-controls-hessian} and \eqref{eq:mass-gt} gives the second equality.
\end{proof}

\subsection{Proof of the main theorem}

\begin{proof}[Proof of Theorem~\ref{thm:main}]
Suppose first that \(n=3\). By Proposition~\ref{prop:equality-reduction}, \(\bar g\) is flat, and \(M\) is simply connected by \cite[Lemma~3.8]{BZ}. Hence there is a developing map
\[
X:(M,\bar g)\longrightarrow\mathbb R^3,
\]
which is a local isometry. 

Now suppose that \(n\ge 4\). As in the proof of Proposition~\ref{prop:drift-coordinates}, write \(X_t^i=\beta Y_t^i+\xi_t^i\), where \(\beta Y_t^i\) is locally uniformly bounded and \(\|\xi_t^i\|_{\mathcal D_0(\bar g)}\le C\). Lemma~\ref{lem:local-Poincare} then gives \(\|X_t^i\|_{L^2(K,\bar g)}\le C_K\) for every \(K\Subset M\). Since \(g_t\to\bar g\) and \(F_t+T_t\to 0\) locally smoothly, interior elliptic estimates give a sequence \(t_j\to 0\) and smooth functions \(X^1,\ldots,X^n\) such that
\[
X_{t_j}^i\longrightarrow X^i\quad\text{in } C_{\mathrm{loc}}^\infty(M).
\]
Since \(m_E(g_{t_j})=2t_j\mathfrak c_E\to 0\), Proposition~\ref{prop:mass-controls-hessian} gives \(\operatorname{Hess}_{\bar g}X^i=0\), so each \(dX^i\) is parallel. Moreover, in the coordinates \(x\) from Lemma~\ref{lem:harmonic-asymptotics}, the bound \eqref{eq:X-asymptotic} gives \(X^i-x^i\in W_\delta^{3,p}(E_R)\). It follows that \(\langle dX^i,dX^j\rangle_{\bar g}=\delta_{ij}\), and hence \(X\) is a local isometry.

Thus, in either case, there is a local isometry \(X\). In the coordinates \(x\),
\[
\delta_{\alpha\beta}\partial_iX^\alpha\partial_jX^\beta=\bar g_{ij},\qquad\partial_i\partial_jX^\alpha=\bar\Gamma_{ij}^k\partial_kX^\alpha=O(r^{-q-1}).
\]
It follows that \(dX=Q+O(r^{-q})\) for some \(Q\in O(n)\), and hence \(|X|/r\to 1\) as \(r\to\infty\). Thus \(X|_{\overline{E_R}}\) is proper for all sufficiently large \(R\). On the other hand, the Schoen--Yau Liouville theorem~\cite{SchoenYau88} (see \cite[Theorem~1.3]{LUY}) shows that \(X\) is injective. Thus the Jordan--Brouwer separation theorem shows that \(X(E_R)\) is the unbounded component of \(\mathbb R^n\setminus X(\partial E_R)\). Consequently, \(S:=\mathbb R^n\setminus X(M)\) is compact.

Via \(X\), we can regard \(V\) as a positive harmonic function on \(\mathbb R^n\setminus S\). Since \(V^{4/(n-2)}g_{\mathrm{Euc}}\) is complete, Karakhanyan's theorem~\cite[Theorem~1.1]{Karakhanyan} gives
\[
\operatorname{Cap}_{1+\frac{2}{n},\frac{n}{2}}(S)=0.
\]
Finally, \cite[Corollary~5.1.14]{AdamsHedberg} yields
\[
\dim_{\mathcal H}S\le n-\left(1+\frac{2}{n}\right)\frac{n}{2}=\frac{n-2}{2}.
\]
\end{proof}


\begin{thebibliography}{99}

\bibitem{AdamsHedberg}
D. R. Adams and L. I. Hedberg,
\emph{Function spaces and potential theory},
Grundlehren Math. Wiss. \textbf{314}, Springer-Verlag, Berlin, 1996.

\bibitem{Bartnik}
R. Bartnik,
\emph{The mass of an asymptotically flat manifold},
Comm. Pure Appl. Math. \textbf{39} (1986), 661--693.

\bibitem{BZ}
Y. Bi and J. Zhu,
\emph{Mass-capacity inequality modeled on conformally flat manifolds},
preprint, 2025, arXiv:2511.13215.

\bibitem{BZpsc}
Y. Bi and J. Zhu,
\emph{Positive scalar curvature obstructions via singular dimension descent},
preprint, 2026, arXiv:2606.20528.

\bibitem{Bray}
H. L. Bray,
\emph{Proof of the Riemannian Penrose inequality using the positive mass
theorem},
J. Differential Geom. \textbf{59} (2001), no. 2, 177--267.

\bibitem{BrayMiao}
H. L. Bray and P. Miao,
\emph{On the capacity of surfaces in manifolds with nonnegative scalar
curvature},
Invent. Math. \textbf{172} (2008), no. 3, 459--475.

\bibitem{BW}
S. Brendle and Y. Wang,
\emph{A dimension descent scheme for the positive mass theorem in arbitrary
dimension},
preprint, 2026, arXiv:2604.08473.


\bibitem{Case}
J. S. Case,
\emph{Smooth metric measure spaces and quasi-Einstein metrics},
Internat. J. Math. \textbf{23} (2012), no.~10, 1250110, 36 pp.

\bibitem{HirschMiaoTam}
S. Hirsch, P. Miao, and L.-F. Tam,
\emph{Monotone quantities of $p$-harmonic functions and their applications},
Pure Appl. Math. Q. \textbf{20} (2024), no. 2, 599--644.

\bibitem{Karakhanyan}
A. L. Karakhanyan,
\emph{Singular Yamabe problem for scalar flat metrics on the sphere},
Manuscripta Math. \textbf{174} (2024), 875--889.

\bibitem{LLU}
D. A. Lee, M. Lesourd, and R. Unger,
\emph{Density and positive mass theorems for incomplete manifolds},
Calc. Var. Partial Differential Equations \textbf{62} (2023),
Paper No. 194.

\bibitem{LUY}
M. Lesourd, R. Unger, and S.-T. Yau,
\emph{The positive mass theorem with arbitrary ends},
J. Differential Geom. \textbf{128} (2024), 257--293.

\bibitem{MiaoImplications}
P. Miao,
\emph{Implications of some mass-capacity inequalities},
J. Geom. Anal. \textbf{34} (2024), Paper No. 241.

\bibitem{MiaoMassCapacity}
P. Miao,
\emph{Mass, capacitary functions, and the mass-to-capacity ratio},
Peking Math. J. \textbf{8} (2025), no. 2, 351--404.

\bibitem{SchoenYau88}
R. Schoen and S.-T. Yau,
\emph{Conformally flat manifolds, Kleinian groups and scalar curvature},
Invent. Math. \textbf{92} (1988), 47--71.

\bibitem{Zhu}
J. Zhu,
\emph{Positive mass theorem with arbitrary ends and its application},
Int. Math. Res. Not. IMRN \textbf{2023} (2023), no. 11, 9880--9900.

\end{thebibliography}
\end{document}